\documentclass[10pt,reqno]{article}
\RequirePackage{amsmath}
\usepackage{amscd,amsthm,amsfonts,amsopn,amssymb,verbatim}
\usepackage{times}
\numberwithin{equation}{section}
\def\be{\begin{equation}}
\def\ee{\end{equation}}

\def\ve{\varepsilon}
\def\nint{\mathop{\diagup\kern-13.0pt\int}}
\def\textnoint{\matpop {\raise1.25pt\hbox{${\ssize\diagup}$}\kern-9.5pt\int}}

\allowdisplaybreaks
\usepackage{amssymb}
\usepackage[dvips]{epsfig}
\newtheorem{theorem}{Theorem} %[section]
\newtheorem{corollary}[theorem]{Corollary}

\theoremstyle{remark}

\begin{document}
\date{\today}
\setcounter{section}{-1}

\title{Decoupling, exponential sums and the Riemann zeta function}
\author{J. Bourgain\thanks 
{This work was partially supported by NSF grants DMS-1301619}\\
Institute for Advanced Study, Princeton, NJ 08540\\
bourgain@math.ias.edu}
\date{\today}
\maketitle

\begin{abstract}{We establish a new decoupling inequality for curves in the spirit of \cite{B-D1},\cite{B-D2} which implies a new mean value
theorem for certain exponential sums crucial to the Bombieri-Iwaniec method as developed further in \cite{H}.
In particular, this leads to an improved bound $|\zeta(\frac{1}{2} + it)| \ll t^{13/84 + \varepsilon}$ for the zeta function on the
critical line.}
\end{abstract}

\section{Introduction}
The main result of the paper is the essentially sharp bound on the mean-value expression for $r = 6$
(see \cite {H} for details)
$$
A_r\Big(\frac{1}{N^2}, \frac{1}{N}\Big) = \int_{0}^{1} \int_{0}^{1} \int_{-1}^{1} \int_{-1}^{1}
 \left|
\sum_{1\le n \leq N} e(nx_1 + n^2x_2+ N^{\frac{1}{2}}n^{3/2}x_3 + N^{\frac{1}{2}}n^{\frac{1}{2}}x_4) \right|^{2 r} dx_1 dx_2 dx_3
dx_4.
\eqno{(0.1)}
$$

It is proven indeed that $A_6 \ll N^{6 + \varepsilon}$ (see Theorem 2 below).
The bound $A_5(\delta, \delta N) \ll \delta N^{7+\varepsilon}, \frac{1}{N^{2}} \leq \delta \leq \frac{1}{N}$, established in \cite{H-K},
plays a key role in the refinement of the Bombieri-Iwaniec approach \cite{B-I1} to bounding exponential sums as developed mainly by Huxley (see \cite{H}
for an expository presentation).
As pointed out in \cite{H}, obtaining good bounds on $A_6$ leads to further improvements and this objective was our main motivation.

In \cite {B}, we recovered the \cite {H-K} $A_5$-result (in fact in a sharper form) as a consequence of certain general decoupling inequalities
related to the harmonic analysis of curves in $\mathbb R^d$.
Those inequalities were derived from the results in \cite{B-D1} (see also \cite {B-D2}).
Theorem 2 will similarly be derived from a decoupling theorem, formulated as Theorem 1.

Let us next briefly recall the basic structure of the Bombieri-Iwaniec argument.
Given an exponential sum \hfill\break
$\sum_{m\sim M} e\big(TF(\frac mM)\big)$ with $T >M$ and $F$ a smooth function satisfying appropriate derivative
conditions, the sum $\sum_{m\sim M}$ is replaced by shorter sums
$\sum_{m\in I}$, $I$ ranging over size-$N$ intervals \big(here $N$ is a parameter to be
chosen and is not the same as in $(0.1)$\big).
For each $I$, the phase may be replaced by a cubic polynomial and, by Poisson summation, the exponential sum $\sum_{m\in I} e\big(TF(\frac
mM)\big)$ transformed (effectively) in a sum of the form
$$
\sum_{h\leq H} e\big(x_1(I)h+x_2(I) h^2+x_3(I) h^{3/2}+ x_4 (I) h^{1/2}\big)\eqno{(0.2)}
$$
where the vector $x(I) =\big(x_j(I)\big)_{1\leq j\leq 4} \in\mathbb R^4$ depends on the interval $I$.

At this point, one needs to analyze the distributions of
$$
(h, h^2, h^{3/2}, h^{1/2}) \quad (1\leq h\leq H) \eqno{(0.3)}
$$
and
$$
\text {the vector function $x(I)$ of the interval $I$} \eqno{(0.4)}
$$
which Huxley refers to as the first and second spacing problems.

Before applying a large sieve estimate, one takes an $r$-fold convolution of (0.3) whose $L^2$-norm is expressed by mean values of the form (0.1) 
with $N$ replaced by $H$.
Roughly speaking, the $L^2$-norm of the distribution (0.4) is bounded by a certain parameter $B$, whose evaluation is highly non-trivial and so far
sub-optimal.
The only input of this paper is to provide an optimal result for the first spacing problem (see Corollary 3 below).
Combined with available treatments of the second spacing problem, it leads to improved exponential sum estimates.

New exponential sum bounds are presented in \S3.
They are based on combining Corollary 3 with known estimates on the parameter $B$ from the second spacing problem (see the Acknowledgement below).

The conclusion is stated as Theorem 4.

\medskip
In \S4, we establish our new estimate on $|\zeta(\frac 12 +it)|$.

\medskip
\noindent
{\bf Theorem 5.}
$$
\Big|\zeta\Big(\frac 12+it\Big)\Big| \ll |t|^{\frac {13}{84}+\ve}.
\eqno{(0.5)}
$$
\medskip

It is implied by the classical approximate functional equation (cf. [T]) together with Theorem 4 and some further (known) exponential sum bounds.

Recall that the original Bombieri-Iwaniec argument provided the estimate
$\big|\zeta\big(\frac 12 + it\big)\big| \ll |t|^{\frac 9{56}+\ve}$, $\frac 9{56} = 0,16071$ (see \cite{B-I1} \cite{B-I2}).
The work of Huxley in \cite {H1} (resp. \cite {H2}) produced the exponents
$$
\frac{89}{570} = 0.15614...  \text{ and } \frac {32}{205} = 0.15609..., \text { resp}
$$
while our $A_6$-bound leads to the exponent $\frac {13}{84}= 0.15476...$, hence doubling the saving over $\frac 16$ obtained in [B-I1].
\medskip

In \S5 we highlight a new exponent pair that results from our work. 

\medskip

\noindent
{\bf Acknowledgement}

The author is grateful to C.~Demeter for various  comments leading to simplification of an earlier version.
Most importantly, he is greatly indebted to one of the referees for clarifying several matters related to the `second spacing problem' and
providing an alternative treatment of Section 3 in the original manuscript that moreover leads to an improved exponent in Theorem 5.
The presentation of the last 3 sections of the paper follows closely his suggestions.

\section
{A decoupling inequality for curves}

Let $\Phi =(\phi_1, \ldots, \phi_d): [0,1] \to\Gamma\subset \mathbb R^d$ be a smooth parametrization of a non-degenerate curve in $\mathbb
R^d$,
more specifically we assume the Wronskian determinant
$$
\det [\phi_j^{(s)} (t_s)_{1\leq j, s\leq d}]\not= 0 \text { for all $t_1, \ldots, t_d\in [0,1]$}.\eqno{(1.1)}
$$
Let us assume moreover that $d$ is even.
For $\Omega\subset \mathbb R^d$ a bounded set of positive measure, denote by 
$$\Vert f\Vert _{L^p_{\#}(\Omega)} = (\nint_\Omega|f|^pdx)^{\frac 1p}=\big(\frac 1{|\Omega|} \int_\Omega |f|^p
dx\big)^{\frac 1p}$$ the average $L^p$-norm and let $B_\rho$ be the $\rho$-cube in $\mathbb R^d$ centered at 0.
We prove the following decoupling property in the spirit of results in \cite {B}, \cite{B-D2}.

\begin{theorem}
Let $\Gamma$ be as above and $I_1, \ldots, I_{\frac d2}\subset [0, 1]$ subintervals that are $O(1)$-separated, let $N$ be large and $\{I_\tau\}$ a
partition of $[0, 1]$ in $N^{-\frac 12}$-intervals.
Then for arbitrary coefficient functions $a_j = a_j (t)$
$$
\begin{aligned}
&\Big\Vert\prod_{j=1}^{d/2} \Big|\int_{I_j} a_j(t) e\big(x.  \Phi (t) \big) dt\Big|^{2/d}\Big\Vert_{L_{\#}^{3d} (B_N)}\ll\\
&N^{\frac 16+\ve} \prod_{j=1}^{d/2} \Big[\sum_{\tau; I_\tau \subset I_j}\Big\Vert\int_{I_\tau} a_j(t) e\big(x.\Phi(t)\big)
dt\Big\Vert^6_{L_{\#}^6(B_N)}\Big]^{\frac 1{3d}}
\end{aligned}
\eqno{(1.2)}
$$
holds, with $\ve>0$ arbitrary.
\end{theorem}

Here $e(z)$ stands for $e^{2\pi iz}$ as usual.
Strictly speaking, $L^6_{\#}(B_N)$ in the right hand side of (1.2) should be some weighted space $L^6_{\#}(w_N)$ with weight $1_{B_N}(x)\lesssim
w_N(x)\leq (1+\frac {|x|}N)^{-10d}$, supp\,$\widehat w_N\subset B_{\frac 1N}$ (cf. \cite{B-D1} and \cite {B-D2}).
For simplicity, this technical point will be ignored here and in the sequel.
\bigskip

Let us say a few words about the role of Theorem 1 in the proof of Theorem 2 stated below. First, the inequality in Theorem 2 can be rewritten as
$$\left(\frac{1}{|\tilde{\Omega}|}\int_{\tilde{\Omega}}|\sum_{n\le N}e(x\cdot\Phi(\frac{n}{N}))|^{12}dx\right)^{1/12}\ll N^{\frac12+\epsilon}$$
with $$
\tilde\Omega =[0, N]\times[0, N^2]\times [0, N^2]\times [0, N]
$$
and
$$\Phi(t)=(t,t^2,t^{3/2},t^{1/2}).$$

Theorem 1 implies via a standard discretization argument that the average over each ball $B_{N^2}$ with radius $N^2$ satisfies
$$\left(\frac{1}{|B_{N^2}|}\int_{|B_{N^2}|}\left[\sqrt{\prod_{j=1}^2|\sum_{n\in I_j}e(x\cdot\Phi(\frac{n}{N}))|}\;\;\right]^{12}dx\right)^{1/12}\ll N^{\frac12+\epsilon}.$$
This estimate is weaker than the one in Theorem 2 in two regards. First, in Theorem 2 one is interested in averages over the smaller region $\tilde{\Omega}$. This issue is dealt with in the first half of Section 2, by using two further standard 2D decouplings and exploiting periodicity. 

The second difference between Theorems 1 and 2 is that the former produces a bilinear estimate, while the latter requires a linear estimate. This issue is addressed in the second part of Section 2, and relies on a variant of the induction on scales from \cite{B-G}.

\bigskip

\noindent
{\bf Remarks.}

\begin{description}
\item [(1.3)]  Obviously (1.2) implies the same inequality for $B_N$ replaced by any translate.

\item
[(1.4)] \ The case $d=2$ is an immediate consequence of te $L^6$-decoupling inequality for planar curves $\Gamma$ of non-vanishing curvature
$$
\Big\Vert \int_0^1 a(t) e\big(x.\Phi (t)\big) dt\Big\Vert_{L^6(B_N)} \ll N^\ve \Big(\sum_\tau \Big\Vert\int_{I_\tau}
a(t) e\big(x.\Phi(t)\big) dt\Big\Vert^2 _{L^6(B_N)}\Big)^{\frac 12}\eqno{(1.5)}
$$
where $\Phi:[0, 1]\to\Gamma \subset \mathbb R^2$ with $|\Phi''|\sim 1$ and $\{I_\tau\}$ as above, established in \cite {B-D1}.
In fact, (1.5) will be the main analytical input required for the proof of (1.2). We mention for future use also the following discrete version of (1.5)
$$\left(\frac{1}{|B_N|}\int_{B_N}|\sum_{n\le N}a_ne(x\cdot\Phi(\frac{n}{N}))|^{6}dx\right)^{1/6}\ll N^{\epsilon}\|a_n\|_{l^2},$$
for each complex coefficients $a_n$.

\item[(1.6)] \ In the language of \cite {B-D1}, \cite {B-D2}, (1.2) may be reformulated as follows.
Let $\Gamma_1, \ldots, \Gamma_{d/2} \subset\Gamma$ be $O(1)$-separated arcs and $f_1, \ldots, f_{\frac d2}\in L^1(\mathbb R^d)$
satisfy supp$\,\widehat{f_j} \subset \Gamma_j+B_{\frac 1N}$.
Denote {$f_\tau =(\hat f|_{\Phi(I_\tau)+B_{\frac 1N}})^\vee$} the Fourier restriction of $f$ to the $\underbrace{\frac 1N\times
\cdots \times \frac 1N}_{d-1}
\times \frac 1{\sqrt N}$ tube \hbox{$\Phi (I_\tau)+B_{\frac 1N}$}.
\end{description}

Then
$$
\Big\Vert\prod^{d/2}_{j=1} |f_j|^{2/d}\Big\Vert_{L_{\#}^{3d}(B_N)} \ll N^{\frac 16+\ve}\prod_{j=1}^{d/2} \Big(\sum_\tau
\Vert f_{j, \tau} \Vert^6_{L^6_{\#} (B_N)}\Big)^{\frac 1{3d}}.\eqno{(1.7)}
$$

\begin{description}
\item[(1.8)] It may be worthwhile to explain the relation between (1.7) and other known decoupling inequalities for curves in $\mathbb R^d$.
\end{description}

Firstly, with $\Gamma$ as above and $\Gamma_1, \ldots, \Gamma_d\subset\Gamma$ $O(1)$-separated, one has a $d$-linear inequality (the analogue of
\cite {B-C-T} for curves)
$$
\Big\Vert\prod^d_{j=1} |f_j|^{\frac 1d}\Big\Vert_{L^{2d}_{\#} (B_N)} \leq c_\Gamma \prod^d_{j=1}
\Big(\sum_\tau\Vert f_{j, \tau}\Vert^2_{L^2_{\#}(B_N)}\Big)^{\frac 1{2d}}.\eqno{(1.9)}
$$

This inequality turns out to be elementary.
Using the fact that the map $I_1\times \cdots \times I_d\to\mathbb R^d:(t_1, \ldots, t_d)\mapsto \Phi(t_1)+\cdots +\Phi(t_d)$ is a diffeomorphism for
$I_1, \ldots, I_d$ $O(1)$-separated by assumption (1.1) and Parseval's theorem, one sees indeed that
$$
\Big\Vert\prod^d_{j=1} \Big|\int_{I_j} a_j(t) e\big(x.\Phi(t)\big)dt \Big|\Big\Vert_{L^2(B_N)} \leq c\prod^d_{j=1} \Vert
a_j\Vert_{L^2(I_j)}.\eqno{(1.10)}
$$
On the other hand, one has the $(d-1)$-linear inequality (see \cite{B-D2})
$$
\Big\Vert \prod^{d-1}_{j=1} |f_j|^{\frac 1{d-1}} \Big\Vert_{L^{2(d+1)}_{\#}(B_N)} \ll N^{\frac 1{2(d+1)}} \prod^{d-1}_{j=1}
\Big[\sum_\tau \Big\Vert f_{j, \tau}\Big\Vert^{\frac {2(d+1)}{d-1}}_{L_{\#}^{\frac{2(d+1)}{d-1}}(B_N)}\Big]^{\frac 1{2(d+1)}}\eqno{(1.11)}
$$
and one observes, for $d$ even, that the pair $\big(2(d+1), \frac {2(d+1)}{d-1)}\big)$ in (1.11) is obtained by interpolation between
the pairs $(2d, 2)$ from (1.9) and $(3d, 6)$ from (1.7).
The issue of what's the analogue of Theorem 1 for odd $d$ will not be considered here.
In fact, our main interest is $d=4$, which provides the required ingredient for the exponential sum application.

Before  passing to the proof of Theorem 1, we make a few preliminary observations.

Note that in the setting of Theorem 1, (1.9) also implies the inequality
$$
\begin{aligned}
&\Big\Vert\prod^{d/2}_{j=1} \Big|\sum_{I_j} a_j(t) e\big(x.\Phi (t)\big) dt\Big|^{2/d}\Big\Vert_{L^d_{\#}(B_N)}\leq\\
&c\prod_{j=1}^{d/2} \Big[\sum_{I_\tau \subset I_j}\Big\Vert \int_{I_\tau} a_j (t) e\big(x.\Phi(t)\big) dt\Big\Vert^2_{L^2_{\#}(B_N)}\Big]^{\frac
1d}.
\end{aligned}
\eqno{(1.12)}
$$
To see this, take $f_j(x)=\frac 1{\sqrt N}\sum_{\substack{0\leq k\leq N\\ \frac kN\in I_j}} \ve_k \, e\big(x.\Phi(\frac kN)\big)$ for $j=\frac d2+1,
\ldots, d$ with $\ve_k=\pm 1$ independent random variables and average over $\{\ve_k\}$, noting that $\mathbb E_\ve [|f_j|^2]
\asymp 1$ and $\mathbb E_\ve [|f_\tau|^2] \asymp  N^{-\frac 12}$.

There is also the trivial bound
$$
\begin{aligned}
&\Big\Vert\prod^{d/2}_{j=1} \Big|\int_{I_j} a_j(t) e\big(x.\Phi(t)\big) dt\Big|^{2/d} \Big\Vert_{L^\infty(B_N)}\leq\\
&N^{\frac 12}\prod^{d/2}_{j=1} \max_{I_\tau \subset I_j}\Big\Vert \int_{I_\tau} a_j(t) e\big(x.\Phi(t)\big)dt
\Big\Vert^{2/d}_{L^\infty (B_N)}.
\end{aligned}
\eqno{(1.13)}
$$

Interpolation between (1.12) and (1.13) using appropriate wave packet decomposition as explained in \cite{B-D1} (note that it is essential here
that the $I_\tau$ are $N^{-\frac 12}$-intervals) gives
$$
\begin{aligned}
&\Big\Vert\prod^{d/2}_{j=1} \Big|\int_{I_j} a_j(t) e\big(x.\Phi(t)\big) dt\Big|^{2/d}\Big\Vert_{L^{3d}_{{\#}}(B_N)} \leq\\
& CN^{\frac 13} \prod^{d/2}_{j=1} \Big[\sum_{I_\tau \subset I_j} \Big\Vert \int_{I_\tau} a_j(t) e\big(x.\Phi (t)\big)
dt\Big\Vert^6_{L^6_{\#}(B_N)} \Big]^{\frac 1{3d}}
\end{aligned}
\eqno{(1.14)}
$$
with $\{I_\tau\}$ a partition in $N^{-\frac 12}$-intervals.

More generally, if $\Delta =\Delta_K\subset\mathbb R^d$ is a $K$-cube, we have (by translation)
$$
\begin{aligned}
&\Big\Vert\prod^{d/2}_{j=1} \Big|\int_{I_j} a_j(t) e\big(x.\Phi(t)\big)dt\Big|^{2/d} \Big\Vert _{L_{\#}^{3d} (\Delta)}\leq\\
&CK^{\frac 13} \prod_{j=1}^{d/2} \big[\sum_{I_\tau\subset I_j} \Big\Vert\int_{I_j} a_j(t)
e\big(x.\Phi(t)\big)dt\Big\Vert^6_{L^6_{\#}(\Delta)}\Big]^{\frac 1{3d}}
\end{aligned}
\eqno{(1.15)}
$$
where $\{I_\tau\}$ is now a partition in $K^{-\frac 12}$-intervals.

The main point of (1.15) is to provide a preliminary $L^6-L^{3d}$ inequality; the prefactor $K^{1/3} $ is not important for
what follows as it will be improved to $K^{\ve+\frac16}$ using a bootstrap argument.

Returning to (1.1), it follows from the mean value theorem that
$$
|\det [\phi_i' (t_j)_{1\leq i, j\leq d}]| \sim\prod_{i\not=j} |t_i-t_j|.\eqno{(1.16)}
$$
By (1.16) and since $\phi''(t)=\lim_{s\to 0} \frac 1s\big(\phi'(t+s)-\phi'(t)\big)$, it follows that for $t_1<\cdots< t_{d/2}\in [0,1]$
$O(1)$-separated,
$$
|\phi'(t_1)\wedge \phi'' (t_1)\wedge \phi'(t_2)\wedge \phi''(t_2)\wedge\cdots\wedge \phi'(t_{\frac d2})
\wedge\phi''(t_{\frac d2})|>c\eqno{(1.17)}
$$
holds.
\bigskip

\noindent
{\bf Proof of Theorem 1.}

Introduce numbers $b(N)>0$ for which the inequality, with arbitrary $\{a_j\}$,
$$
\begin{aligned}
&\Big\Vert\prod^{d/2}_{j=1} \Big|\int_{I_j} a_j(t) e\big(x.\Phi(t)\big)dt\Big|^{2/d} \Big\Vert_{L_{\#}^{3d}(B_N)}\leq\\
&b(N) N^{\frac 16} \prod^{d/2}_{j=1} \Big[\sum_{I_\tau \subset I_j} \Big\Vert\int_{I_\tau} a_j(t) e\big(x.\Phi(t)\big)
dt\Big\Vert^6_{L^6_{\#}(B_N)}\Big]^{\frac 1{3d}}
\end{aligned}
\eqno{(1.18)}
$$
holds.
Our aim is to establish a bootstrap inequality.
By (1.14), $b(N)\leq N^{1/6}$.
With $K<N$ to specify, partition $B_N$ in $K$-cubes $\Delta =\Delta_K$.
We may bound for each $\Delta$ (since the inequalities for $B_K$ and $\Delta_K$ are equivalent)
$$
\begin{aligned}
&\nint_\Delta \prod_{j=1}^{d/2} \Big|\int_{I_j} a_j(t) e\big(x.\Phi(t)\big)dt\Big|^6dx\leq\\
&b(K)^{3d} K^{\frac d2} \prod_{j=1}^{d/2} \Big[\sum_{I_\sigma \subset I_j} \Big\Vert\int_{I_\sigma} a_j(t)
e\big(x.\Phi(t)\big)dt\Big\Vert^6_{L^6_{\#}(\Delta)}\Big]
\end{aligned}
\eqno{(1.19)}
$$
with $\{I_\sigma\}$ a partition in $K^{-\frac 12}$-intervals.
Summation over $\Delta\subset B_N$ implies then
$$
\begin{aligned}
&\nint_{B_N} \prod_{j=1}^{d/2} \Big|\int_{I_j} a_j(t) e\big(x.\Phi(t)\big) dt\Big|^6 dx\leq\\
&b(K)^{3d} K^{\frac d2} \sum_{I_{\sigma_1}\subset I_1, \ldots, I_{\sigma_{d/2}}\subset I_{d/2}}
\nint_{B_K^{d/2}}\Big\{\nint_{B_N} \prod_{j=1}^{d/2} \Big|\int_{I_{\sigma_j}}a_j(t) e\big((x+z_j).\Phi(t)\big) dt\Big|^6 dx\Big\} \prod_j dz_j.
\end{aligned}
\eqno{(1.20)}
$$
Fix $I_{\sigma_j} =[t_j, t_j+K^{-\frac 12}]\subset I_j$ and write for $t=t_j+s\in I_{\sigma_j}$
$$
(x+z_j).\Phi(t)=(x+z_j). \Phi(t_j)+(x+z_j). \Phi'(t_j)s+ \frac 12(x+z_j).\Phi'' (t_j) s^2+o(1)\eqno{(1.21)}
$$
provided
$$
N=o(K^{3/2}).\eqno{(1.22)}
$$
The inner integral in (1.20) may then be replaced by
$$
\nint_{B_N} \prod_{j=1}^{d/2}\Big|\int_0^{K^{-\frac 12}} a_j(t_j+s) e\big((x+z_j).\Phi'(t_j)s+\frac 12 (x+z_j).\Phi'' (t_j)s^2\big)
ds\Big|^6dx\eqno{(1.23)}
$$
the $o(1)$-term in (1.21) producing a harmless smooth Fourier multiplier that may be ignored.

Next, since $t_1<t_2<\cdots < t_{d/2}$ are $O(1)$-separated, (1.17) applies and therefore the map $\mathbb R^d\to\mathbb R^d: x\mapsto
\big(x.\Phi'(t_1), \frac 12 x. \Phi''(t_1), \ldots, x.\Phi'(t_{d/2}), \frac 12 x.\Phi'' (t_{d/2})\big)$ is a linear homeomorphism.
The image measure of the normalized measure on $B_N$ may be bounded by the
normalized measure on $B_{CN}$, up to a factor and
$$
(1.23) \lesssim \prod_{j=1}^{d/2} \nint_{|u|, |v|<CN} \Big|\int_0^{K^{-\frac 12}}
a_j (t_j+s) \, e(us+vs^2) ds\Big|^6 dudv.\eqno{(1.24)}
$$
This factorization is the main point in the argument.

We may now apply (after rescaling $s=k^{-\frac 12}s_1)$ to each factor
in (1.24) the $2D$-decoupling inequality (1.5) with $\Gamma$ the
parabola $(s_1, s_1^2)$ and perform a decoupling at scale $(\frac
KN)^{\frac 12}$.
Thus, by another change of variables,
$$
(1.24) \ll N^\ve \prod_{j=1}^{d/2} \Big[\sum_{I_\tau \subset
I_{\sigma_j}} \Big\Vert\int_{I_\tau} a_j(t) \, e(ut) dt\Big\Vert^2_{L^6_{\#} [|u|<CN]}\Big]^3
$$
with $\{I_\tau\}$ a partition in $N^{-\frac 12}$-intervals
$$
\ll N^\ve\Big(\frac NK\Big)^{\frac d2} \prod_{j=1}^{d/2}
\Big[\sum_{I_\tau \subset I_{\sigma_j}} \Big\Vert \int_{I_\tau} a_j(t) \, e\big(x.\Phi(t)\big)
dt\Big\Vert^6_{L^6_{\#}(B_N)}\Big].\eqno{(1.25)}
$$
Substituting (1.25) in (1.20) leads to the estimate
$$
b(K)^{3d} N^{\frac d2+\ve} \prod^{d/2}_{j=1} \Big[\sum_{I_\tau \subset
I_j}
\Big\Vert\int_{I_\tau} a_j(t)
\, e\big(x.\Phi(t)\big)dt\Big\Vert^6_{L^6_{\#} (B_N)}\Big].\eqno{(1.26)}
$$
Recalling (1.22), one may conclude that
$$
b(N) \ll b(N^{2/3})N^\ve
$$
and Theorem 1 follows by iteration.

\section
{A mean value theorem}

From now on, we focus on $d=4$ (in view of the application to exponential sums) and consider $\Phi:[0, 1]\to \Gamma\subset\mathbb R^4$
satisfying (1.1).
If $I_1, I_2 \subset\{1, \ldots, N\}$ are $\sim N$ separated, we get from Theorem 1
$$
\begin{aligned}
&\Big\Vert \prod^2_{j=1} \Big|\sum _{n\in I_j} a_n \, e\Big(\Phi\Big(\frac nN\Big).x\Big)
\Big|^{\frac 12} \Big\Vert_{L^{12}_{\#}(B_N)}\ll\\
&N^{\frac 16+\ve} \prod^2_{j=1} \Big(\sum_{J\subset I_j}\Big\Vert \sum_{n\in J} a_n e\Big(\Phi\Big(\frac nN\Big).x\Big)\Big\Vert^6_{L^6_{\#}(B_N)}\Big)^{\frac 1{12}}
\end{aligned}
\eqno{(2.1)}
$$
with $\{J\}$ a partition of $\{1, \ldots, N\}$ in $N^{\frac 12}$-intervals.

Again in view of the application, specify
$$
\phi_1(t)=t, \phi_2(t)=t^2\eqno{(2.2)}
$$
and assume
$$
|\phi_3'''|>c.\eqno{(2.3)}
$$
In order to perform a further decoupling in (2.1), we enlarge the domain $B_N$, considering first
$$
\Omega= [0, N]\times [0, N^{3/2}]\times [0, N^{3/2}] \times [0, N]
$$
which we partition in $N$-cubes $\Delta_N$.

Let $I_1, I_2$ be as above.
Application of (2.1) on $\Delta_N$ gives
$$
\begin{aligned}
&\Big\Vert \prod^2_{j=1} \Big|\sum_{n\in I_j} a_n \, e\Big(\Phi\Big(\frac nN\Big).x\Big)\Big|^{\frac
12}\Big\Vert_{L^{12}_{\#}(\Delta_N)}\ll\\
&N^{\frac 16+\ve}\Big[\prod^2_{j=1} \Big(\sum_{J\subset I_j}\Big\Vert\sum_{n\in J} a_n \, e\Big(\Phi\Big(\frac nN\Big).x\Big)
\Big\Vert^6_{L^6_{\#}(\Delta_N)}\Big)\Big]^{\frac 1{12}}
\end{aligned}
$$
and summing over $\Delta_N$
$$
\begin{aligned}
&\Big\Vert \prod^2_{j=1} \Big|\sum_{n\in I_j} \cdots \Big|^{\frac 12} \Big\Vert_{L_{\#}^{12}(\Omega)}\ll\\
&N^{\frac 16+\ve}\Big[\sum_{\substack{J_1\subset I_1\\ J_2\subset I_2}} \nint_{B_N\times B_N} dzdz' \nint_\Omega dx
\Big|\sum_{n\in J_1} a_n \, e\Big(\Phi\Big(\frac nN\Big).(x+z)\Big)\Big|^6 \Big|\sum_{n\in J_2} a_n \, e\Big(\Phi\Big(\frac
nN\Big)(x+z')\Big)\Big|^6\Big]^{\frac 1{12}}.
\end{aligned}
\eqno{(2.4)}
$$
Let $J_1=[h_1, h_1+N^{\frac 12}], J_2=[h_2, h_2+N^{\frac 12}]$ with $h_1-h_2\asymp N$.
Write for $n\in J_1$, $n=h_1+m$, recalling (2.2)
$$
\begin{aligned}
\Phi\Big(\frac nN\Big).(x+z) = & \Phi\Big(\frac {h_1}N\Big).(x+z)+\\
&\frac mN\Big(x_1+z_1+2\frac {h_1}N(x_2+z_2)+\phi_3' \Big(\frac {h_1}N\Big)(x_3+z_3)+\phi_4' \Big(\frac {h_1}N\Big)(x_4+z_4)\Big)+\\
&\frac {m^2}{N^2}\Big( x_2+\frac 12 \phi_3'' \Big(\frac {h_1}N\Big)x_3\Big)+O(1)
\end{aligned}\eqno{(2.5)}
$$
recalling that $|z|,|x_1|,|x_4|<N$ and $|x_2|, |x_3|<N^{3/2}$ while $|m|< N^{\frac 12}$.
Proceed similarly for \hfill\break
$\Phi\Big(\frac nN\Big).(x+z'), n\in J_2$.

Observe that $z_1, z_1'$ have range $[0, N]$, so that periodicity considerations and a change of variables in $z_1, z_1'$ permit to
replace the phase (2.5) by
$$
\frac mNz_1 +\frac {m^2}{N^2}\Big( x_2+\frac 12 \phi_3''\Big(\frac{h_1}N\Big)x_3\Big)
$$
and
$$
\frac mN z_1' +\frac {m^2}{N^2} \Big(x_2+\frac 12 \phi_3'' \Big(\frac {h_2}{N}\Big)x_3\Big).
$$
Since $h_1-h_2\asymp N$ and (2.3), one more change of variables in $x_2, x_ 3$ gives the phases
$$
\begin{cases}
mu_1+\frac {m^2}{N^{1/2}}w_1\\ mu_2 +\frac {m^2}{N^{1/2}}w_2\end{cases}
\eqno{(2.6)}
$$
with $u_1, u_2, w_1, w_2$ ranging in $[0,1]$. Hence we obtain again a factorization
of the integrand in (2.4), i.e.
$$
\begin{aligned}
\int_0^1\int_0^1\int_0^1\int_0^1 du_1 du_2 dw_1 dw_2 &\Big|\sum_{m_1<\sqrt N}
a_{h_1+m_1}\,  e\Big({m_1u_1}+ \frac {m_1^2}{N^{1/2}}w_1\Big) \Big|^6\\
&\Big|\sum_{m_2<\sqrt N} a_{h_2+m_2} \, e\Big(m_2u_2+\frac {m^2_2}{N^{1/2}}w_2\Big)\Big|^6
\end{aligned}
\eqno{(2.7)}
$$
and the $2D$-decoupling result applied to each factor enables to make a further decoupling at scale $N^{1/4}$.
This clearly permits to bound (2.4) by
$$
N^{\frac 16+\ve}N^{\frac 1{12}+\ve}\Big[\sum_{\substack{J_1'\subset I_1\\ J_2'\subset I_2}} \int_0^1\int^1_0 \Big|\sum_{n\in J_1'}
a_n \, e(nu_1)\Big|^6 \Big|\sum_{n\in J_2'} a_n \, e(nu_2)\Big|^6 
du_1 du_2\Big]^{\frac 1{12}}\eqno{(2.8)}
$$
with $\{J'\}$ a partition in $N^{\frac 14}$-intervals. The fact that we have dropped the term $\frac{n^2}{N^{1/2}}w_i$ that appears in (2.7) deserves a word of explanation. Note that if $n\in J_i'$ then $n=h_i+m$ with $m\le N^{1/4}$. 
Thus $\frac{n^2}{N^{1/2}}w=\big(\frac{h^2}{N^{1/2}}+\frac{2hm}{N^{1/2}}\big)w+O(1)$, which permits us to replace in (2.7) the argument by
$m\big(u+\frac{2h}{N^{1/2}} w\big)$ and hence $mu$ by change of variable.

If instead we consider a translate $\Omega+y$ of $\Omega$, the expression  (2.8) needs to be modified replacing $a_n$
by $a_n\, e\big(\Phi(\frac nN).y\big)$.

Finally, consider the domain (according to Huxley's $A_6$-problem)
$$
\tilde\Omega =[0, N]\times[0, N^2]\times [0, N^2]\times [0, N]
$$
which we partition in domains $\Omega_\alpha =\Omega+y_\alpha$ with $\Omega$ as above.
Thus for each $\alpha$ (2.8) implies
$$
\begin{aligned}
&\Big\Vert\prod^2_{j=1} \Big|\sum_{n\in I_j}\cdots \Big|^{\frac 12}\Big\Vert_{L^{12}_{\#}(\Omega_\alpha)}\ll\\
&N^{\frac 14+\ve}\Big[\sum_{\substack{J_1'\subset I_1\\ J_2'\subset I_2}} \int_0^1\int^1_0 \Big|\sum_{n\in J_1'} a_n
\, e\Big(\Phi\Big(\frac nN\Big).y_\alpha\Big) e(nu_1)\Big|^6 \, \Big|\sum_{n\in J_2'}\cdots \Big|^6 du_1 du_2\Big]^{\frac 1{12}}
\end{aligned}
$$
and
$$
\begin{aligned}
&\Big\Vert\prod^2_{j=1} \Big|\sum_{n\in I_j}\cdots \Big|^{\frac 12} \Big\Vert_{L^{12}_{\#}(\tilde\Omega)} \ll\\
&N^{\frac 14+\ve} \Big[\sum_{\substack {J_1'\subset I_1\\ J_2'\subset I_2}}\nint_{\tilde{\Omega}}\int^1_0\int^1_0\Big|\sum_{n\in J_1'}
a_n \, e\Big(\Phi\Big(\frac nN\Big).y+nu_1\Big)\Big|^6 \, \Big|\sum_{n\in J_2'} \cdots \Big|^6 dydu_1du_2\Big]^{\frac 1{12}}.
\end{aligned}
\eqno{(2.9)}
$$

Proceeding as before, let $J_1'=[h_1, h_1+N^{\frac 14}], J_2' =[h_2, h_2+N^{\frac 14}], h_1-h_2\asymp N$.

Write for $n\in J_1', n=h_1+m$
$$
\begin{aligned}
&\Phi\Big(\frac nN\Big).y+nu_1=\\
&\Phi\Big(\frac {h_1}N\Big).y+h_1u_1+\\
&\qquad m\Big(u_1+ \frac {y_1}N+ 2\Big(\frac {h_1}N\Big)y_2+\frac 1N\phi_3'\Big(\frac {h_1}N\Big) y_3+\frac 1N\phi_4'\Big(\frac {h_1}N\Big)
y_4\Big)\\
&\qquad +\frac {m^2}{N^2} \Big(y_2+\frac 12\phi_3''\Big(\frac{h_1} N\Big) y_3\Big) +O(1).
\end{aligned}
$$
Since $|y_2|, |y_3|<N^2, |y_4|<N \text { and } |m|<N^{\frac 14}$.

Again by periodicity, (2.3) and change of variables,  we obtain the phases
$$
mu_1+m^2 w_1
$$
and
$$
mu_2+m^2w_2
$$
with $u_1, u_2, w_1, w_2 \in [0, 1]$ and the $L^6$-norms are bounded by the $\ell^2$-norms of the coefficients (see Remark
 1.4).
In conclusion, we proved that
$$
\Big\Vert\prod^2_{j=1} \Big|\sum_{n\in I_j} a_n \, e\Big(\Phi\Big(\frac nN\Big).x\Big)\Big|^{\frac 12} \Big\Vert_{L^{12}_{\#} (\tilde\Omega)}
\ll N^{\frac 12+\ve}\Vert\bar a\Vert_\infty\eqno{(2.10)}
$$
with $\Phi$ satisfying (1.1), (2.2), (2.3), i.e.
$$
\phi_1(t) =t, \phi_2(t) =t^2, |\phi_3''|>c \text { and } \Big|\left|\begin{matrix} \phi_3'''(s)&\phi_4'''(s)\\ \phi_3''''(t) &
\phi_4'''' (t)\end{matrix}\right|\Big| >c \text { for } s, t\in [0, 1].\eqno{(2.11)}
$$

The following statement is the mean value estimate for $A_6$ in \cite{H}.

\begin{theorem}
$$
\int_0^1\int_0^1\int_{-1}^1\int_{-1}^1\Big|\sum_{n\leq N}
e(nx_1+n^2x_2+N^{\frac 12} n^{3/2}x_3+N^{\frac 12} n^{\frac 12} x_4)\Big|^{12} dx_1 dx_2dx_3dx_4\ll N^{6+\ve}.\eqno{(2.12)}
$$
\end{theorem}

\begin{proof}

Let $I\subset [1, N]$ be an interval of the form $[N_0, N_0+M]$, $100 M<N_0\leq N$, and assume $I_1, I_2\subset I$ subintervals of size $\sim
M$ that are $\sim M$-separated.
\end{proof}

We first estimate
$$
\int\Big\{\prod^2_{j=1} \Big|\sum_{n\in I_j} e(nx_1+n^2x_2+N^{1/2} n^{3/2} x_3+N^{1/2}n^{1/2} x_4)\Big|^6\Big\} dx.\eqno{(2.13)}
$$
Clearly (2.13) amounts to the number of solutions of the system
$$
\left\{\!\!\!\!
\begin{array}{lll}
&m_1+m_2+m_3-m_4-m_5-m_6 = m_7+m_8+m_9- m_{10}-m_{11}-m_{12}&{(2.14)}\\
&m_1^2+m_2^2+m_3^2-m^2_4-m_5^2-m_6^2 = m^2_7+m_8^2+m_9^2-m_{10}^2-m_{11}^2-m_{12}^2&{(2.15)}\\
&(N_0+m_1)^{3/2}+(N_0+m_2)^{3/2}+(N_0+m_3)^{3/2}-(N_0+m_4)^{3/2}-(N_0+m_5)^{3/2}-(N_0+m_6)^{3/2}= &{(2.16)}\\
&(N_0+m_7))^{3/2}+(N_0+m_8)^{3/2}+(N_0+m_9)^{3/2}-(N_0+m_{10})^{3/2} -
(N_0+m_{11})^{3/2}-(N_0+m_{12})^{3/2}+O(N^{-\frac 12})\\
&(N_0+m_1)^{\frac 12}+\cdots -(N_0+m_6)^{1/2}=\\
&(N_0+m_7)^{\frac 12}+\cdots- (N_0+m_{12})^{\frac 12}+O(N^{-\frac 12}).&{(2.17)}
\end{array}
\right.
$$
with $m_1, \ldots, m_6 \in I_1' =I_1-N_0; m_7, \ldots, m_{12} \in I_2' =I_2-N_0$.

Write $(N_0+m)^{3/2}, (N_0+m)^{1/2}$ in the form
$$
(N_0+m)^{3/2} =N_0^{3/2}+\frac 32 N_0^{\frac 12}m+\frac 38 N_0^{-\frac 12} m^2 +M^3 N_0^{-3/2} \phi_3 \Big(\frac mM\Big)
\qquad\qquad\qquad\qquad\eqno{(2.18)}
$$
$$
(N_0+m)^{1/2}=N_0^{1/2}+\frac 12 N_0^{-\frac 12}m-\frac 18 N_0^{-\frac 32} m^2-M^3 N_0^{-5/2} \phi_3 \Big(\frac mM\Big)+M^4 N_0^{-7/2} \phi_4
\Big(\frac mM\Big)\eqno{(2.19)}
$$
where $\phi_3(t) \sim t^3\Big(1+O\Big(\frac M{N_0}\Big)t+\cdots\Big), \phi_4(t)\sim t^4$.

Hence $\Phi(t)= \big(t, t^2, \phi_3(t), \phi_4(t)\big)$ satisfies (2.11).

From (2.14), (2.15), (2.18), (2.19), inequalities (2.16), (2.17) may be replaced by
$$
\phi_3\Big(\frac {m_1}M\Big)+\cdots+\phi_3 \Big(\frac {m_{12}}M\Big)<O(N^{-\frac 12}N_0^{3/2} M^{-3})\eqno{(2.20)}
$$
$$
\phi_4\Big(\frac{m_1}M\Big)+\cdots+\phi_4 \Big(\frac {m_{12}}M\Big)<O(N^{-\frac 12}N_0^{7/2} M^{-4}).\eqno{(2.21)}
$$
The number of solutions of (2.14), (2.15), (2.20), (2.21) may be evaluated by
$$
\int_{[-1,1]^4}\Big\{\prod_{j=1}^2\Big|\sum_{m\in I_j'} e(mx_1+m^2x_2+\frac {N^{\frac 12} M^3}{N_0^{3/2}} \phi_3\Big(\frac mM\Big) x_3 +
\frac {N^{\frac 12} M^4} {N_0^{7/2}} \phi_4 \Big(\frac mM\Big) x_4\Big)\Big|^6\Big\} dx.\eqno{(2.22)}
$$
According to (2.10), (2.22) and hence (2.13) are bounded by
$$
M^{6+\ve} \Big\{ 1+\frac {N_0^{3/2}}{N^{\frac 12}M}\Big\} \, \Big\{ 1+\frac {N_0^{7/2}}{N^{1/2}M^3}\Big\}\ll N^{4+\ve}M^2.\eqno{(2.23)}
$$
Returning to (2.12), let $b (N)N^6$ be a bound on the left hand side
We use the same reduction procedure to multi-linear (here bi-linear) inequalities as in \cite{B}, \cite{B-D2} (and originating from \cite{B-G}).
Denote $K$ a large constant and partition $[0, N]$ in intervals $I_0, I_1, \ldots, I_K$, where $|I_0|= \frac {100N}K$ and $|I_s|=\big( 1-\frac
{100}K\big) \frac NK=M_0$ for $1\leq s\leq K$.

Bound
$$
\int\Big|\sum_{n\leq N}\Big|^{12} \leq 2^{12} \int\Big|\sum_{n\in I_0} \Big|^{12} +(2K)^{12} \sum_{1\leq s\leq K}\int \Big|\sum_{n\in I_s}
\Big|^{12}.\eqno{(2.24)}
$$
The first term of (2.24) is bounded by $2^{12} 100^6 K^{-6} b\Big(\frac {100N}K\Big) N^6$.

For the remaining terms, write $I_s=[N_s, N_s+M_0]$, $N_s>100 M_0$, and make a further partition of $I_s$ in consecutive intervals $I_{s,
1},\ldots, I_{s, K}$ of size $M_1=\frac{M_0}K$.
The key point (going back to \cite {B-G}) is an estimate of the from
$$
\int\Big|\sum_{n\in I_s} \Big|^{12} \leq 4^{12} \sum_{s_1\leq K}\int \Big|\sum_{n\in I_{s, s_1}} \Big|^{12}
+K^{18} \sum_{\substack{ s_1, s_2 \leq K\\ |s_1-s_2|\geq 2}} \int\Big\{\Big|\sum_{n\in I_{s, s_1}} \Big|^6\, \Big|\sum_{n\in I_{s, s_2}}\Big|^6\Big\}.\eqno{(2.25)}
$$
Recall that (2.25) follows from considering the (pointwise in $x$) decreasing rearrangement $\eta_1\geq \eta_2\geq\cdots\geq \eta_K$ of the
sequence $\big(\big|\sum_{n\in I_{s, s_1}}\big|\big)_{1\leq s_1\leq K}$ and distinguishing the cases $\eta_4 <\frac 1{K_1}\eta_1$ and
$\eta_4\geq \frac 1{K_1}\eta_1$.
\medskip

Application of (2.23) gives for $|s_1-s_2|\geq 2$
$$
\int\Big\{\Big|\sum_{n\in I_{s, s_1}}\Big|^6 \, \Big|\sum_{n\in I_{s, s_2}}\Big|^6\Big\}\ll N^{6+\ve}
$$
and hence the second sum in (2.25) contributes at most for $C(K)N^{6+\ve}$.
Replace the second term in the r.h.s. of (2.24) by
$$
(2K)^{12} 4^{12} \sum_{s\leq K, s_1\leq K}\int\Big|\sum_{n\in I_{s, s_1}} \Big|^{12}.
$$
Repeating the procedure, partition each $I_{s, s_1}$ in intervals $I_{s, s_1, s_2}$ of size $M_2=\frac {M_1}K$ and apply the decomposition
(2.25) for each $\sum_{n\in I_{s, s_1}}$ etc.

In general, one gets bilinear contributions of the form
$$
(2K)^{12} 4^{12\alpha}K^{18} \sum_{J, J'} \int\Big\{ \Big|\sum_{n\in J} \Big|^6 \, \Big|\sum_{n\in J'}\Big|^6\Big\}.
\eqno{(2.26)}
$$
where the sum extends over pairs $J, J'$ of intervals of size $M_\alpha =\frac N{K^{\alpha+1}}, \alpha\geq 1$ that are at least $M_\alpha$-separated
and contained in an interval of the form $[N_0, N_0+KM_\alpha], KM_\alpha < \frac 1{100} N_0$.
Again by (2.23)
$$
\int\Big\{\Big|\sum_{n\in J}\Big|^6\, \Big|\sum_{n\in J'} \Big|^6\Big\} \ll N^{4+\ve} M^2_\alpha
$$
implying that
$$
(2.26) \ll C(K) 4^{12\alpha} \frac N{M_\alpha} N^{4+\ve} M_\alpha^2 \ll N^{6+\ve}\Big(\frac {4^{12}}K\Big)^{\alpha}.
$$
Summing over $\alpha$ eventually leads to the bound
$$
2^{12} 100^6 K^{-6} b\Big(\frac {100N}K\Big)N^6+N^{6+\ve}.\eqno{(2.27)}
$$
On the l.h.s. of (2.24). Therefore
$$
b(N)\leq 2^{12} 100^6K^{-6} b\Big(\frac {100N}K\Big) +C_\ve N^\ve
$$
implying $b(N)\ll N^\ve$ and Theorem 2.
\medskip

Using the notation from \cite{H}, Theorem 2 implies

\begin{corollary}
Let $\frac 1{N^2} \leq\delta \leq 1, \frac 1N\leq\Delta\leq 1$.
Then
$$
A_6(N, \delta, \Delta)=\int_0^1 \int_0^1\int^1_{-1}\int^1_{-1} 
\Big|\sum_{n\leq N} e\Big(nx_1+ n^2x_2+\frac 1\delta \Big(\frac nN\Big)^{3/2} x_3 +\frac 1\Delta\Big(\frac nN\Big)^{1/2}
x_4\Big)\Big|^{12} dx\ll \delta\Delta N^{9+\ve}.\eqno{(2.28)}
$$
\end{corollary}

Considering the major arc contribution, (2.28) is clearly seen to be essentially best possible.

\section
{\bf Applications to exponential sums}

Let $F$ be a smooth function on $[\frac 12, 1]$ satisfying, for some constant $c\in (0, 1]$,  the condition
$$
\min\{|F''(x)|, |F'''(x)|, |F^{''''} (x)|\}>c.\eqno{(3.1)}
$$
Given $T$ sufficiently large, $M\geq 1$, put : $ f(u)=TF (u/M)$ with $\frac M2\leq u\leq M$ and
$$
S=\sum_{m\sim M} e\big(f(m)\big). \eqno{(3.2)}
$$
In what follows, we assume $M\leq \sqrt T$, in view of the application to $|\zeta(\frac 12+it)|$.
We use notation and background from \cite {H} and also rely on [H-W] and \S 7 and \S 8 in [H1].

Once the parameter $N\in (1, M)$ is chosen,  $R$ is defined by the relation
$$
R=\Big\lceil\Big( \frac {2M^3}{cNT}\Big)^{1/2}\Big\rceil,\eqno{(3.3)}
$$
so that, for each relevant size-$N$ interval $I\subset [M/2, M]$, the corresponding `arc' $J(I)=\{f''(u)/2: u\in I\}$ will be an interval of length exceeding $1/R^2$.
We assume that $N$ and $R$ satisfy the conditions
$$
R\leq N\leq R^2
$$
(given that $\sqrt T\geq M$, these conditions imply that $2\leq R\leq N\ll M^{3/2}T^{-1/2}\ll MT^{-1/4}\leq M^{1/2}$ and $N\gg MT^{-1/3}\gg 1$).
By following (in all but certain inessential respects) the steps of \S4 in [H-W] that precede Equation (4.8) there, while slightly modifying the application of the
lemma on `partial sums by Fourier transforms' (i.e. [H-W], Lemma 3.4), one obtains a result implying that, for some $Q, \ell, H, \alpha\in\mathbb C$ satisfying
$$
Q, H\in\mathbb N, \ell \in \{0, 1, 2\}, \alpha \in \{e(-\eta): -1/2 \leq \eta\leq 1/2\}, Q\geq R\gg \sqrt Q \ 
\text { and } \ H\geq NQ/R^2\gg H,
$$ 
one has an upper bound
$$
|S| \ll \frac {M\log N}{N^{1/2}}+ \frac {R\log^2 N}{Q^{1/2}}\sum_{I\in \mathcal I(Q, \ell)}  \Big(\Big|\sum_{h\leq H}
\alpha^h e \big({\bf x}(I)\cdot \big(h, h^2, h^{3/2}, h^{1/2}\big)\big)\Big|+\frac QR\Big)\eqno{(3.4)}
$$
in which ${\bf x}$ is a certain mapping from the set $\mathcal I(Q, \ell)$ into $[-X_1, X_1]\times \cdots \times [-X_4, X_4]\subset \mathbb R^4$, where
$$
X_1= X_2 =\frac 12 \ \text { and } \ X_3 =X_4 =\Big(\frac RQ\Big)^2 H^{1/2},
$$
while $\mathcal I(Q, \ell)$ is the set of those $I\in \{[kN, N+kN]: k\in\mathbb N$ and $M/(2N)\leq k\leq (M-3N)/N\}$ that, via the  procedures set out in [H-W], \S4,
Step 1, are associated with a reduced rational $a/q\in J(I)$ that happens to satisfy $Q\leq q< 2Q$ and $Q\equiv \pm \ell (\text{mod\,} 4)$ (in the terminology of [H-W], \S4, Step 1, these $I$'s are `minor arcs').
The details of the `second spacing problem' are not amongst the main points of interest in this paper, so we skip the definition of 
${\bf x} (I)$ and mention only that
this element of $\mathbb R^4$ is essentially identical to the vector ${\bf y}={\bf y}^{(k)}$ defined in [H-W], \S4 Step 4 (the index $`k'$ there corresponds to our `$I$').

By an appropriate application of [H-W], Lemma 2.1, one finds that
$$
\sum^2_{\ell=0}|\mathcal I(Q, \ell)| \ll \frac {MR^2}{NQ^2}+\frac {R^2} Q\asymp \frac {MR^2}{NQ^2}.
$$
Given this estimate, that in (3.4) and the trivial upper bound for the modulus of the sum over $h$ in (3.4), it follows by the sixth-power H\"older inequality that,
either
$$
|S|\ll \frac {M\log ^2N}{N^{1/2}}\eqno{(3.5)}
$$
or else
$$
R\leq Q< R^{2/3} N^{1/3}\leq N\eqno{(3.6)}
$$
and one has
$$
|S|^6\ll \Big(\frac {R\log^2 N}{Q^{1/2}}\Big)^6 \Big(\frac {MR^2}{NQ^2}\Big)^5\sum_{I\in\mathcal I(Q, \ell)} \sum_{{\bf h}\in \mathbb N^6} e\big(
{\bf x}(I)\cdot {\bf y}({\bf h})\big) \omega (I)\Omega ({\bf h}), \eqno{(3.7)}
$$
where
$$
{\bf y}({\bf h})=\sum^6_{j=1} (h_j, h_j^2, h_j^{3/2}, h_j^{1/2})= (h_1+\cdots+h_6, \ldots, h_1^{1/2}+\cdots + h_6^{1/2})\in \mathbb R^4,
$$
while $\omega$ is a certain complex-valued function that takes values that are (without exception) of modulus not exceeding unity, as
does $\Omega({\bf h})$ 
 (which is equal to $\alpha^{y_1({\bf h})}$ if $h_j \leq H$ for $j=1, \ldots, 6$, and is zero otherwise).

Suppose now that (3.6) and (3.7) hold.
Then, similarly to what is observed at the end of `Step 4', in [H-W], \S4, it follows by the Bombieri-Iwaniec `double large sieve' [B-I1], Lemma 2.4 (or see [H-W],
Lemma 3.6) that one has
$$
\Big|\sum_{I\in\mathcal I(Q, \ell)} \,  \sum_{{\bf h}\in \mathbb N^6} e\big({\bf x}(I)\cdot {\bf y}({\bf h})\big)\omega(I) \Omega({\bf h})
\Big|^2\ll AB_1 \prod^4_{j=1} (X_j Y_j+1).\eqno {(3.8)}
$$
where
$$
\begin{aligned}
&Y_1=6H, \qquad Y_2=6H^2, \qquad Y_3=6H^{3/2}, \qquad Y_4=6H^{1/2},\\
B_1= &\big|\big\{I, I'): I, I'\in\mathcal I(Q, \ell) \ \text { and } \ |x_j(I)-x_j(I')|<\frac 1{2Y_j} (j=1, \ldots, 4)\big\}\big|
\end{aligned}
\eqno{(3.9)}
$$
and
$$
A=\Big|\big\{ ({\bf h}, {\bf h}'):{\bf h}, {\bf h}'\in(\mathbb N\cap (0, H])^6 \ \text { and } \ |y_j({\bf h})-y_j({\bf h}')|
< \frac 1{2X_j} (j=1, \ldots, 4)\big\}\Big|.
$$
It is worth remarking here that the conditions on $({\bf h}, {\bf h}')$ in the above definition of the number $A$ actually imply the
equality of the ordered pairs $\big(y_1({\bf h}), y_2({\bf h})\big)$ and $\big(y_1({\bf h}'), y_2({\bf h}')\big)$ (both of which lie in $\mathbb Z^2$): hence the traditional definition of the `first spacing problem' as a question concerning the order of magnitude of the number of
solutions in integers $h_1, h_1',\ldots, h_r, h_r'\in (0, H]$ of a certain system of two equations and two inequalities \big(see for example [H], (11.1.1)-(11.1.5)\big).
Given that $|y-y'|<\frac 1{2X}$ implies $|y-y'|<\frac 1X$ (whenever $X, y, y'\in \mathbb R)$, it is a corollary of [H], Lemma 5.6.5 (for example) that we have here:
$$
\begin{aligned}
0\leq A&\leq \Big(\frac {\pi^2}2\Big)^4 \Big(\frac 1{X_1\cdots X_4}\Big) \int^{X_1}_{-X_1} \int^{X_2}_{-X_2}\int^{X_3}_{-X_3} \int_{-X_4}^{X_4}\Big|
\sum_{\substack{{\bf h}\in \mathbb N^6\\ h_j\leq H(j=1, \ldots, 6)}} e({\bf y}({\bf h}) \cdot 
(z_1, \ldots, z_4)\big)\Big|^2 dz_1dz_2dz_3dz_4\\
&=\frac {\pi^8Q^4}{4HR^4}\int^1_0\int^1_0 \int^{\frac {R^2\sqrt H}{Q^2}}_{-\frac {R^2\sqrt H}{Q^2}} 
\int^{\frac {R^2\sqrt H}{Q^2}}_{-\frac{R^2\sqrt H}{Q^2}}\Big|
\sum_{h\leq H} e\big(( h, h^2, h^{3/2}, h^{1/2}) \cdot (z_1, \ldots, z_4)\big)\Big|^{12} dz_1dz_2dz_3dz_4\\
&=(\pi^8/4)A_6(H;\delta, H\delta).
\end{aligned}
$$
with $A_6(L; \gamma, \Gamma)$ defined according to (2.28), and with
$$
\delta =\frac {Q^2}{H^2R^2}
$$
so that $1/H\leq H\delta\leq Q/N\leq 1$.
Therefore it follows by Corollary 3 that we have
$$
A\ll \delta^2 H^{10+\ve}.\eqno{(3.10)}
$$
With regard to the `second spacing problem' one uses the treatment of the second spacing problem in [H-W], rather than the more advanced treatment in [H1]).
One obtains
$$
B_1\ll \Delta_1\Delta_2 \Big(\frac MN\Big)^2 \Big(\frac QR\Big)^4 
\text { if }  \ N=MT^{-2/7}
$$
where
$$
\Delta_1=\frac 1{X_2Y_2}= \frac 1{3H^2}<\frac {R^4}{N^2Q^2} \ \text { and } \ \Delta_2 =\frac 1{X_3Y_3} =\frac {Q^2}{6R^2 H^2}<\frac {R^2}{N^2}.
$$
Hence, with $N=MT^{-2/7}$
$$
B_1\ll \frac {M^2 R^2Q^2}{N^6}\eqno{(3.11)}
$$

It follows from (3.5)-(3.8), (3.10) and (3.11) that
$$
|S|^6\ll \max \Big\{ \frac {M^{6+\ve}}{N^3}, \Big(\frac {M^{5+\ve} R^4N}{Q^7} \Big ) \Big(\frac {MRQ}{N^3}\Big)\Big\} \leq 
 \Big(\frac {M^{6+\ve}}{N^3}\Big)  \Big(\frac NR\Big) \ \text { for } N=MT^{-2/7}.\eqno{(3.12)}
$$
Recalling (3.3) the above bound for $|S|^6$ implies
$$
|S|^6\ll\Big(\frac {M^{6+\ve}}{N^3}\Big)\Big(\frac {N^3T}{M^3}\Big)^{1/2} =\frac {M^{\ve+9/2} T^{1/2}}{N^{3/2}}= M^{3+\ve} T^{13/14},
$$
If $\sqrt T\geq M\geq c T^{3/7}$ (where $c$ is the positive constant in (3.1) and (3.3)) then the conditions $N\in (1, M)$ and $R\leq N\leq R^2$ are satisfied, with
$N=MT^{-2/7}$ and $R$ as in (3.3).
That is, we have:
$$
|S|\ll M^{\frac 12} T^{\ve+13/84} \  \text { for } \ \sqrt T\geq M\geq cT^{3/7}.\eqno{(3.13)}
$$

The $M$-range for which the bound (3.13) holds may be extended by invoking the treatment in [H1] which we discuss next.

It was observed by Huxley, at the start of \S7 in [H1], that for an arbitrary $V\geq 1$ the structure of the Bombieri-Iwaniec `double large sieve' implies that if the factor $X_2Y_2+1$
on the right-hand side of the bound (3.8) is increased to $X_2Y_2V+1\leq (X_2Y_2+1)V$ then the adjacent term $B_1$ may be replaced by a term $B_V\leq B_1$, the definition of which
differs from that of $B_1$ \big(in (3.9)\big) only insofar as it involves an upper bound on $|x_2(I)-x_2(I')|$ that is stronger by a factor $V$ than is the case in (3.9).
This observation plays a crucial part in Huxley's method of `resonance curves', through which the most recent progress [H1], [H], [H2] on the `second spacing problem' was achieved;
we apply it here, in combination with (3.10) and the bounds $\prod_{j\leq 4}(X_jY_j+1)=\hfill\break
(3H+1) (3H^2+1) (6\delta^{-1}+1) (6(H\delta)^{-1}+1)\ll H^2/\delta^2$, in order to deduce that
$$
|S|^6\ll \Big(\frac {M^5R^{16} \log^{12}N}{N^5Q^{13}}\Big) H^{6+\ve}(VB_V)^{1/2}\ll \Big(\frac {M^5R^4N^{1+\ve}}{Q^7}\Big) (VB_V)^{1/2} \quad (V\geq 1).\eqno{(3.14)}
$$

In [H1] Huxley invented an approach to the `second spacing problem' based on a theory involving certain `resonance curves'.
In his first application of this, in [H1], \S7, he obtained a result implying that, if $M\leq \sqrt T$ (as we suppose), and if one has either $V=N/Q\ll R^4/N^2$, or else $V=R^4/N^2$
(so that $V\geq 1$ in either case, given that $Q\leq N\leq R^2)$, then
$$
VB_V\ll \Big(\frac {VMR^2}{NQ^2}+\Delta_1\Delta_2\Delta_4^{2/3} \Big(\frac MN\Big)^2\Big) \Big(\frac QR\Big)^4\eqno{(3.15)}
$$
where $\Delta_1, \Delta_2$ are as above, while
$$
\Delta_4=\frac 1{X_4Y_4}= \frac {Q^2}{6R^2H}<\frac QN.
$$
By (3.14) and (3.15), one obtains:

$$
\begin{aligned}
|S|^6&\ll \Big(\frac{M^{5+\ve} R^4N}{Q^7}\Big)\Big(\min\Big\{\Big(\frac NQ\Big)^{1/2}, \frac {R^2}N\Big\} 
\Big(\frac MN\Big)^{1/2}+\Big(\frac QR\Big)^{1/3} \Big(\frac RN\Big)^{7/3}\Big(\frac MN\Big)\Big)\\
&\leq \Big(\frac {M^{5+\ve}N}{R^3}\Big)\Big(\min\Big\{\Big(\frac NR\Big)^{1/2}, \frac {R^2}N\Big\}\Big(\frac MN\Big)^{1/2} +\Big(\frac RN\Big)^{7/3} \Big(\frac MN\Big)\Big)\\
&=\min \Big\{\frac {M^{\ve+11/2}N}{R^{7/2}} , \frac{M^{\ve+11/2}}{RN^{1/2}}\Big\}+\Big(\frac {M^{6+\ve}}{N^3}\Big)\Big(\frac NR\Big)^{2/3}
\end{aligned}
\eqno{(3.16)}
$$
\big(with $M^{\ve+11/2}/(RN^{1/2})\asymp M^{4+\ve} T^{1/2}$, by virtue of (3.3)\big). 
When $M>T^{11/30}$ and $R$ is given by (3.3), the upper bound (3.16) may be optimized by putting $N=\max \{MT^{-17/57}, M^{1/2}T^{-1/12}\}$: for each such $M, T, N$ and $R$, the bound
(3.16) is equivalent to:
$$
|S|^6 \ll \Big(\frac{M^{6+\ve}}{N^3}\Big)\Big(\frac NR\Big)^{2/3}\asymp \frac{M^{5+\ve}T^{1/3}}{N^2}.\eqno{(3.17)}
$$
One can check that if $T$ is sufficiently large (in terms of $c^{-1})$, if $\sqrt T\geq M\geq T^{5/12}$, and if $N$ is as assumed in (3.17), then $N$ does satisfy our initial assumptions
\big(that $1< N<M$ and $R\leq N\leq R^2$, with $R$ given by (3.3)\big).
In fact our optimal choice of $N$ in connection with the application of (3.15) is, unsurprisingly, identical to the choice of $N$ found to be optimal in [H1], \S8.
The same is true in cases where $M<T^{5/12}$.
The problem with such cases is that the choice of $N$ assumed in (3.17) is too large to satisfy the condition $N\leq R^2$.
The solution to this problem (utilized in [H1]) is to switch to a smaller value of $N$ satisfying $N\leq R^2\ll N$.
This is achieved here by putting $N=\big(2M^3/(cT)\big)^{1/2}$, which satisfies all of our assumptions considering $N$ and $R$ whenever $T^{5/12}\geq M\geq 2(cT)^{1/3}$.
It moreover follows from (3.16) and (3.3) that one obtains (3.17) for this alternate choice of $N$ (satisfying $N\asymp R^2$).

By the above observations, and the further observation that $MT^{-17/57}\gg M^{1/2}T^{-1/12}$ if and only if $M\gg T^{49/114}$, we arrive at the upper bounds.
$$
|S|^6 \ll \begin{cases}  M^{3+\ve}T^{53/57} \ & \text { if } \ \sqrt T\geq M > T^{49/114};\\
M^{4+\ve}T^{1/2} &\text { if } T^{49/114}\geq M\geq T^{5/12};\\
M^{2+\ve} T^{4/3} &\text { if } T^{5/12} > M\geq   2(cT)^{1/3}.
\end{cases}\eqno{(3.18)}
$$
Note that this bound is $\min\{M^{3+\ve} T^{53/57}, M^{4+\ve}T^{1/2}\}$ when $\sqrt T\geq M\geq T^{5/12}$.

Using (3.18) if $M<T^{3/7}$ (noting the inequalities $\frac {49}{114}>\frac 37 >\frac 5{12}$), one verifies that the bound on $|S|$ in (3.13) holds if 
$\sqrt T\geq M\geq T^{17/42}$.
Thus we establish
\medskip

\noindent
{\bf Theorem 4.}
{\sl With the above notation, one has that
$$
|S|\ll M^{1/2} T^{\ve+13/84} \  \text { if }  \  \frac 12 \geq \alpha =\frac {\log M}{\log T} \geq \frac {17}{42}.\eqno{(3.19)}
$$}

\section
{Bounding the zeta-function on the critical line}

For the application to $|\zeta(1/2+ it)|$ one wants to show that (3.19) also holds when $17/42>\alpha\geq 0$.
The cases with $0\leq \alpha \leq 13/42$ are trivial (there one can just use $|S|\leq M)$, so all that remains to be done is establishing
that (3.19) holds when $\alpha$ lies in the interval $(13/42, 17/42)$.
To achieve this one can employ the bound
$$
|S|\ll T^{\frac 1{128} (4+103\alpha)+\ve} \qquad (12/31 <\alpha\leq 1),\eqno {(4.1)}
$$
which is [H1], Theorem 3, in combination with the exponent pair estimate
$$
|S|\ll \Big(\frac TM\Big)^{1/9} M^{13/18} =M^{11/18} T^{1/9} \quad (0\leq \alpha \leq 1),\eqno{(4.2)}
$$
which corresponds to the exponent pair $(\frac 19, \frac {13}{18})= ABA^2B(0, 1)$ mentioned in [T], \S5 20.
It should be noted that (4.2) \big(and also (4.1)\big) assume additional hypotheses concerning the function $F$, beyond condition (3.1).
This, however, is not an obstacle to the application to $|\zeta(1/2+it)|$, since that only requires consideration of cases in which $F(x)=\log x$ \big(a function that does
satisfy all the unmentioned conditions attached to (4.1) and (4.2)\big).
Assume henceforth that $F$ is ``a suitable function''. such that (4.1) and (4.2) are applicable.
A calculation shows that (3.19) is implied by (4.1) for all $\alpha$ in the interval $(12/31, 332/819]$, and is implied by (4.2) for all $\alpha$ in the interval 
$[0, 11/28]$: noting that
$11/28=0.39285 \ldots > 0.38709\ldots = 12/31$, we find that the union of these two intervals is $[0,332/819]=[0,0.40537\ldots]\supset (0.30952\ldots, 0.40476\ldots)
=(13/42, 17/42)$.
\medskip

By the preceding one has the bound (3.19) whenever $0\leq \alpha\leq 1/2$ (at least this is so in the case $F(x) =\log x)$.
It follows from the `approximate functional equation' for $\zeta (s)$ in the critical strip, see [T] (4.12.4),
that
$$
\Big|\zeta\Big(\frac 12+it\Big)\Big|\leq 2\Big|\sum_{n\leq \sqrt {t/2\pi} } n^{-\frac 12+it} \Big|+ O(1) \ (t\to \infty).\eqno{(4.3)}
$$
From partial summation and dyadic dissection, Theorem 5 now follows in the usual way.

\medskip
\noindent
{\bf Remark.}

Bombieri and Iwaniec achieved the exponent $9/56 =(1-1/28)/6$ using an essentially optimal bound on $A_4(\delta, H\delta)$, so that the
exponent $13/84 =(1-1/14)/6$ \big(achieved with an essentially optimal  bound for $A_6(\delta, H\delta)$\big)
represents exactly a doubling of Bombieri and Iwaniec's improvement over the classical `1/6' (with their essentially optimal bound on $A_5 (\delta, H\delta)$ Huxley and Kolesnik, in 1991,
got the exponent $17/108=(1-1/18)/6$, and would in fact have got $11/70=(1-2/35)/6$, except for the fact that cases with $\alpha$ near to 23/54 were a problem at that time).

\section
{Further comments}

Recalling (3.2), the preceding shows that one has the estimate
$$
|S|\ll M^{\frac 12} T^{\ve+\frac {13}{84}} \ \text { if } \ \frac 12 \geq \alpha =\frac {\log M}{\log T}>0\eqno{(5.1)}
$$
provided $f$ is in the class of functions to which the exponent pair theory applies (see for instance [G-K], Ch. 3 for details).
In fact

\medskip
{\bf Theorem 6.}
{\sl $\big(\ve+\frac{13}{84}, \ve +\frac {55}{84}\big)$ is an exponent pair.}
\medskip

One needs to obtain the bound $|S|\ll (T/M)^{\frac {13}{84}+\ve} M^{\frac {55}{84}+\ve}= M^{\frac 12} T^{\frac {13}{84}+\ve}$ subject to
conditions that are weaker, in two respects, than the conditions under which direct application of (5.1) gives this bound on $|S|$.
More specifically

\begin{itemize}
\item [(a)] the relevant `$M$' may exceed the square root of the relevant `$T$' (although one will at least not have $M>T$);

\item[(b)] the summation may not be over $[M/2, M]\cap\mathbb Z$ 
\big (it may just be over some set $[a, b]\cap \mathbb Z$, where $[a, b]$ is some proper subset of $(M/2, M)\big)$.
Moreover the function $f$ that one is `given' may only be defined on the subinterval $[a, b]$ (this is, for example, what occurs in 
the theory of exponent pairs developed in [G-K]).
\end{itemize}

As a first step to getting around the problem (a), one can note that in the absence of problem (b) the desired result in any cases with 
$T<M^2\ll T$ can be seen to follow from (5.1).
For in such cases one may replace $c, F, T$ by $c_1= TM^{-2}c, F_1=TM^{-2} F$ and $T_1=M^2$ without invalidating (3.1) or causing any change in
the value of the sum $S$.

Secondly, just to secure any extreme cases, one can deal with the cases in which $M\geq T^{9/10}$ (say) simply by an appeal to the exponent pair $(1/2, 1/2)$ (this is 
analogous to using the trivial bound $|S|\leq M$ when $\alpha$ is in a neighborhood of 0).

Cases with $T^{1/2}<M< T^{9/10}$ become manageable after they are converted, through Poisson summation and partial summation, into cases involving (in place of $M$ and $T$)
an $M'\asymp T/M$ and a $T' \asymp T$, so that one has $(T')^{1/10}\ll M^1 \ll (T')^{1/2}$.
The details of this conversion are essentially the same as what goes on in the verification of the `$B$-process' of exponent pair theory (see, for example, [G-K], \S3.5); its efficacy,
in disposing of problem (a), is related to the fact that, when $0\leq k\leq \frac 12$ and $l=k+\frac 12$, 
one has $B(k, l):=(l-\frac 12, k+\frac 12)=(k, l)$.

Problem (b) is also remarked upon in Sargos's paper [S] (see the remark on p. 310).
Firstly one constructs a suitable extension of the function $f(x)$, so that the resulting function $f_1(x)$
has domain $[M/2, M]$, is identical to $f(x)$ on the subinterval $[a, b]$, and satisfies (on $[M/2, M])$ the requisite set of conditions on its derivatives (these
conditions being such as to make the theory of exponent pairs applicable).
If $s>0$ and $yx^{-s}$ is the relevant `monomial' approximation to $f'(x)$ on $[a, b]$ (such as must be present when the exponent pair
theory is applicable), then it is enough to consider an extension $f_1$ of $f$ that, for $b<x\leq M$, satisfies $f_1(x)= y\int_b^x u^{-s}
du+a_0+a_1 x+\cdots + a_Px^P$, where $a_0, a_1, \ldots, a_P$ are certain constants (determined by the requirement that $f_1^{(P)}(x)$ be
continuous at $x=b$): given that the derivatives $f'(x), \ldots, f^{(P)} (x)$ satisfy the requisite conditions on the interval $[a, b]$ (for
which see [G-K], Condition (3.3.3)) one may deduce that the constants $a_1, \ldots, a_P$ are small enough not to prevent those same
conditions being satisfied by $f_1'(x), \ldots, f_1^{(P)} (x)$ on the longer interval $[a, M]$.
By a similar construction one obtains an extension of $f$, (and so also of $f$) that has domain $[M/2, M]$ and is of the class to which the
exponent pair theory applies.

Once the extension of $f$ to $\big[\frac M2, M\big]$ is obtained, one can employ [S], lemma 2.1 to solve problem (b) at the cost of 
losing a harmless factor $O(\log M)$ in the final estimate.

\end{document}